\documentclass[12pt]{amsart}

\usepackage[utf8]{inputenc}
\usepackage[T1]{fontenc}
\usepackage[margin=1in]{geometry}
\usepackage{amsmath,amssymb,amsthm,amsfonts,mathrsfs}
\usepackage{hyperref}
\usepackage{xcolor}
\usepackage{enumitem}
\usepackage{mathtools}

\hypersetup{
  colorlinks=true,
  linkcolor=blue,
  citecolor=blue,
  urlcolor=blue
}

\newtheorem{theorem}{Theorem}[section]
\newtheorem{proposition}[theorem]{Proposition}
\newtheorem{lemma}[theorem]{Lemma}
\newtheorem{corollary}[theorem]{Corollary}
\newtheorem{definition}[theorem]{Definition}
\newtheorem{remark}[theorem]{Remark}
\newtheorem{example}[theorem]{Example}

\newcommand{\A}{\mathbb{A}}
\newcommand{\C}{\mathbb{C}}
\newcommand{\GM}{\mathbb{G}_{\mathrm{m}}}
\newcommand{\GA}{\mathbb{G}_{\mathrm{a}}}
\newcommand{\Der}{\operatorname{Der}}
\newcommand{\Aut}{\operatorname{Aut}}
\newcommand{\Spec}{\operatorname{Spec}}
\newcommand{\LND}{\operatorname{LND}}
\newcommand{\SSD}{\operatorname{SSD}}
\newcommand{\rank}{\operatorname{rank}}
\newcommand{\ML}{\operatorname{ML}}
\newcommand{\id}{\mathrm{id}}

\title{On nice $\GM$-actions arising from locally nilpotent
       derivations with slice}

\author{Luis Cid}
\address{Instituto de Matem\'atica y F\'isica, Universidad de
  Talca, Casilla 721, Talca, Chile}
\email{luis.cid@inst-mat.utalca.cl}

\date{\today}


\begin{document}

\begin{abstract}
Let $k$ be an algebraically closed field of characteristic zero
and $B$ a finitely generated $k$-domain.  Given a locally nilpotent
derivation $D$ on $B$ admitting a slice $s$, the derivation
$\partial=NsD$ ($N\in\mathbb{Z}\setminus\{0\}$) is semisimple and
defines a regular $\GM$-action on $\Spec(B)$.  We show that this
derivation provides a new explicit description of the $\GM$-action
introduced by Freudenburg in terms of the infinitesimal generator
$\partial=NsD$.  In the nice case ($D^2(x_i)=0$ for all
generators), we prove a linearizability criterion: the associated
$\GM$-action is linearizable if and only if $D$ is automorphically
conjugate to $\frac{\partial}{\partial x_n}$ and the slice becomes
affine-linear in the distinguished variable; moreover, this
criterion is independent of the choice of slice.
\end{abstract}

\maketitle

\section{Introduction}

Let $k$ be an algebraically closed field of characteristic zero,
$B$ a finitely generated $k$-domain, and $X=\Spec(B)$.  Locally
nilpotent derivations (LNDs) on $B$ correspond to regular
$\GA$-actions on $X$.  When a LND admits a \emph{slice}, i.e.,
an element $s\in B$ with $D(s)=1$, the Slice Theorem gives
\[
  B = \ker(D)[s],
\]
making derivations with slice a natural testing ground for
questions on automorphisms and algebraic group actions.

Starting from a LND $D$ with slice $s$ we consider the derivation
\[
  \partial = NsD,\quad N\in\mathbb{Z}\setminus\{0\}.
\]
Since $s$ has eigenvalue $N$ for $\partial$ and $\ker(D)$ is
pointwise fixed, $\partial$ is the \emph{multiplicative
counterpart} of the additive action defined by $D$.  A related
construction appears in Freudenburg~\cite{Fre06}; one of our goals
is to identify his action with the one defined by $\partial$.

We are especially interested in the \emph{nice} case, where
$D^2(x_i)=0$ for every generator $x_i$ of $B$.  Freudenburg asked
whether every nice $\GM$-action arising from a LND with slice is
linearizable.  We give a precise criterion answering this question.
Throughout Sections~\ref{sec:linearization}--\ref{sec:applications}
we work with $B=k[x_1,\dots,x_n]$; the associated $\GM$-action
$\alpha$ on $\A^n=\Spec(B)$ is defined by
$\alpha_t(x_i)=x_i-(1-t^N)s\,D(x_i)$ in the nice case
(see Corollary~\ref{cor:associated-Gm-action} for the general
formula).

\medskip
\noindent\textbf{Main Theorem}
(Theorem~\ref{thm:linearizability-criterion}).
\emph{The following are equivalent:
\begin{enumerate}
\item[\textup{(1)}] The $\GM$-action $\alpha$ is linearizable.
\item[\textup{(2)}] There exists $\phi\in\Aut(B)$ with
  $\phi D\phi^{-1}=\frac{\partial}{\partial x_n}$ and $\phi(s)=x_n$.
\item[\textup{(3)}] There exist $\phi\in\Aut(B)$ and
  $p\in k[x_1,\dots,x_{n-1}]$ with
  $\phi D\phi^{-1}=\frac{\partial}{\partial x_n}$ and
  $\phi(s)=x_n+p$.
\end{enumerate}}
\medskip

The Main Theorem reduces Freudenburg's question entirely to the
study of nice LNDs with slice: the associated $\GM$-action is
linearizable if and only if the underlying nice LND is
automorphically conjugate to a partial derivative
$\frac{\partial}{\partial x_n}$.  This reformulation makes it clear that
the obstruction to linearizability, if any, lies entirely in the
automorphism group of $B$ and not in the structure of the
$\GM$-action itself.

Beyond this criterion we establish the following results.
The \textbf{independence of the slice}
(Theorem~\ref{thm:slice-independence}) shows that if $s$ and $s'$
are two slices of $D$, the corresponding $\GM$-actions are conjugate
by an explicit automorphism, so linearizability does not depend on
the choice of slice.  A \textbf{geometric reformulation}
(Corollary~\ref{cor:kernel-criterion}) states that the action is
linearizable if and only if $\ker(D)$ maps onto a coordinate
hyperplane ring and the slice maps to a translate of the
complementary variable.  For the \textbf{resolution in small
dimensions}, the Main Theorem yields a complete answer for $n\le 3$:
for $n=1$ the result is immediate since every LND with slice on
$k[x]$ is already $\frac{\partial}{\partial x}$; for $n=2$
(Corollary~\ref{cor:dim2}), using Rentschler's
theorem~\cite[Theorem~5.1.5]{Fre06} on LNDs of $k[x,y]$ and the Jung--van der Kulk
theorem~\cite[Theorem~5.1.10]{Fre06}, every nice LND with slice on
$k[x,y]$ is conjugate to $\frac{\partial}{\partial y}$ and the
associated $\GM$-action is linearizable; for $n=3$
(Corollaries~\ref{cor:dim3-wang} and~\ref{cor:wang-general}),
Wang's algebraic classification~\cite{Wan05} implies directly that
every nice LND with slice in $k[x,y,z]$ is conjugate to
$\frac{\partial}{\partial z}$, and over $k=\C$ the same conclusion
follows independently from the theorem of Kaliman, Koras,
Makar-Limanov, and Russell~\cite{KKMR97,KR97}; Wang's result also
covers the larger class $D^2x=D^2y=0$.  Finally, the
\textbf{obstruction via the Makar-Limanov invariant}
(Proposition~\ref{prop:ML-obstruction}) provides a concrete
computational tool for $n\ge 4$: if one can exhibit a regular
function on $\Spec(B)$ invariant under every $\GA$-action but not
a scalar, then $\ML(B)\supsetneq k$ and $\alpha$ cannot be
linearizable.

The article is organized as follows.  Section~\ref{sec:prelim}
recalls the necessary background.  Section~\ref{sec:semisimple}
constructs the semisimple derivation associated to a LND with
slice.  Section~\ref{sec:freudenburg} identifies it with
Freudenburg's action.  Section~\ref{sec:linearization} proves the
Main Theorem.  Section~\ref{sec:applications} contains all further
results, unifying the resolution for $n\le 3$ and the
additive--multiplicative comparison.

\section{Preliminaries}\label{sec:prelim}

Throughout, $k$ denotes an algebraically closed field of
characteristic zero and $B$ a finitely generated $k$-domain.

\subsection{Locally nilpotent and semisimple derivations}

\begin{definition}
$D\in\Der_k(B)$ is \emph{locally nilpotent} if for every $b\in B$
there exists $n\ge 1$ with $D^n(b)=0$.  We write $\LND(B)$ for
the set of all such derivations.
\end{definition}

\begin{definition}
$D\in\Der_k(B)$ is \emph{semisimple} if
$B=\bigoplus_{\lambda\in k}B_\lambda$ where
$B_\lambda=\{b\in B\mid D(b)=\lambda b\}$.  We write $\SSD(B)$
for the set of all such derivations.
\end{definition}

\begin{remark}
If $D\in\SSD(B)$ then $\ker(D)=B_0$.
\end{remark}

\begin{remark}
We denote by $\SSD^\ast(B)\subset\SSD(B)$ the subset of
semisimple derivations with integer eigenvalues.  These are
precisely the derivations corresponding to regular $\GM$-actions
on $\Spec(B)$ (see Remark~\ref{rem:Gm-grading} below), and they
are the central objects of this article.
\end{remark}

\subsection{Slices and the Slice Theorem}

\begin{definition}
Let $D\in\LND(B)$.  An element $s\in B$ is a \emph{slice} for $D$
if $D(s)=1$.
\end{definition}

\begin{theorem}[Slice Theorem; {\cite[Theorem~1.3.21]{Fre06}}]
\label{thm:slice}
If $D\in\LND(B)$ admits a slice $s$ then $B=\ker(D)[s]$.
\end{theorem}

\begin{remark}\label{rem:slice-family}
If $s$ is a slice for $D$ then $s+a$ is also a slice for every
$a\in A:=\ker(D)$.  Conversely, any two slices differ by an
element of $A$.  Thus the set of slices of $D$ is the coset
$s+A$.
\end{remark}

\begin{remark}
Under the isomorphism $B\cong A[s]$ given by the Slice Theorem,
$D$ coincides with $\frac{d}{ds}$: for $b=\sum a_i s^i$ with
$a_i\in A$ one has $D(b)=\sum i\,a_i s^{i-1}$.
\end{remark}

\subsection{Algebraic group actions}

\begin{remark}
A LND $D\in\LND(B)$ determines a regular $\GA$-action on
$X=\Spec(B)$ via the exponential map
$\exp(tD)(b)=\sum_{j\ge 0}(t^j/j!)D^j(b)$, which is a finite
sum for every $b\in B$.
\end{remark}

\begin{remark}\label{rem:Gm-grading}
If $D\in\SSD^\ast(B)$, the decomposition
$B=\bigoplus_{m\in\mathbb{Z}}B_m$ (where $D(b)=mb$ for $b\in B_m$)
determines a regular $\GM$-action on $X$.  Concretely, the action
is obtained from the formal flow of $D$: for $\tau$ a formal
parameter and $t=e^\tau=\sum_{j\ge 0}\tau^j/j!$, one has
\[
  \alpha_t(b) = \exp(\tau D)(b)
  = \sum_{j\ge 0}\frac{\tau^j}{j!}D^j(b)
  = t^m b
  \quad\text{for } b\in B_m.
\]
Conversely, every regular $\GM$-action arises this way, and the
semisimple derivation $D$ is recovered as the infinitesimal
generator: $D=\mathrm{ev}_1\circ\tfrac{d}{dt}\circ\alpha^*$,
where $\alpha^*\colon B\to B[t^{\pm 1}]$ is the coaction;
see~\cite[Definition~2.5 and Theorem~2.14]{CiLi22}.
The invariant ring is $B^{\GM}=B_0=\ker(D)$.
\end{remark}

\section{The semisimple derivation associated with a slice}
\label{sec:semisimple}

Let $D\in\LND(B)$ admit a slice $s$, set $A:=\ker(D)$, so
$B=A[s]$ by Theorem~\ref{thm:slice}.  Fix $N\in\mathbb{Z}
\setminus\{0\}$ and define $\partial:=NsD$.

\begin{proposition}\label{prop:associated-semisimple}
The derivation $\partial=NsD$ is semisimple with
\[
  \partial(a)=0\;(a\in A),\quad \partial(s)=Ns,\quad
  \partial(as^m)=Nm\,as^m\;(a\in A,\,m\ge 0).
\]
\end{proposition}

\begin{proof}
Since $A=\ker(D)$ we have $\partial(a)=NsD(a)=0$.  Since
$D(s)=1$ we get $\partial(s)=Ns$.  For $a\in A$ and $m\ge 0$,
using $D=\frac{d}{ds}$ on $A[s]$,
\[
  \partial(as^m)=Ns\cdot mas^{m-1}=Nm\,as^m.
\]
Every element of $B=A[s]=\bigoplus_{m\ge 0}As^m$ is a finite
sum of such monomials, so $\partial$ is semisimple.
\end{proof}

\begin{corollary}\label{cor:associated-Gm-action}
The derivation $\partial$ defines a regular $\GM$-action on
$\Spec(B)$ given by
\[
  \alpha_t(a)=a\;(a\in A),\qquad
  \alpha_t(P(s))=P(t^N s)\;(P\in A[T]).
\]
\end{corollary}

\begin{proof}
By Proposition~\ref{prop:associated-semisimple}, the eigenvalues
of $\partial$ on $B=A[s]$ are the integers $Nm$, $m\ge 0$.  In
particular all eigenvalues are integers, so by
Remark~\ref{rem:Gm-grading} the derivation $\partial$ defines a
regular $\GM$-action.  Since $\alpha_t$ acts on the weight-$Nm$
space by multiplication by $t^{Nm}$, we get
$\alpha_t(as^m)=t^{Nm}as^m=a(t^N s)^m$ for all $a\in A$,
$m\ge 0$, which gives the stated formulas.
\end{proof}

\begin{remark}
$\ker(\partial)=A=\ker(D)$: indeed, $\partial(as^m)=Nm\,as^m=0$
if and only if $m=0$ or $a=0$, so the zero-eigenspace consists
precisely of the elements of $A=As^0$.
\end{remark}

\begin{example}
Let $D=\frac{\partial}{\partial x_n}\in\LND(k[x_1,\dots,x_n])$,
$s=x_n$.  Then $\partial=Nx_n\frac{\partial}{\partial x_n}$ is
semisimple and
$t\cdot(x_1,\dots,x_n)=(x_1,\dots,x_{n-1},t^N x_n)$.
\end{example}

\section{Relation with Freudenburg's action}\label{sec:freudenburg}

With notation as above, Freudenburg~\cite[Theorem~10.30]{Fre06}
defines a $\GM$-action by
\[
  \beta_t := \exp(-\lambda D)\big|_{\lambda=(1-t^N)s},
\]
where $\lambda$ is an indeterminate over $B$ and $D$ is extended
by $D(\lambda)=0$.

\begin{lemma}\label{lem:translation-formula}
For every $P(T)\in A[T]$,\quad
$\exp(-\lambda D)(P(s))=P(s-\lambda)$.
\end{lemma}

\begin{proof}
Using $D|_A=0$, $D(s)=1$, and the binomial theorem,
\[
  \exp(-\lambda D)(s^m)
  =\sum_{i=0}^m\binom{m}{i}(-\lambda)^i s^{m-i}
  =(s-\lambda)^m.\qedhere
\]
\end{proof}

\begin{theorem}\label{thm:Freudenburg-coincides}
Freudenburg's $\GM$-action $\beta$ coincides with $\alpha$:
\[
  \beta_t(a)=a\;(a\in A),\qquad
  \beta_t(P(s))=P(t^N s)\;(P\in A[T]).
\]
\end{theorem}

\begin{proof}
By Lemma~\ref{lem:translation-formula}, evaluating at
$\lambda=(1-t^N)s$:
$\beta_t(P(s))=P(s-(1-t^N)s)=P(t^N s)$.
\end{proof}

\begin{corollary}\label{cor:infinitesimal-generator}
The infinitesimal generator of Freudenburg's $\GM$-action is
$\partial=NsD$.
\end{corollary}

\begin{proof}
Differentiating $\beta_t(s)=t^N s$ at $t=1$ gives $Ns$, and
differentiating $\beta_t(a)=a$ at $t=1$ gives $0$ for every
$a\in A$.  Since $B=A[s]$ is a polynomial algebra over $A$ in
the single variable $s$, a derivation of $B$ that is $A$-linear
is entirely determined by its value on $s$.  As $\partial$ is
$A$-linear (it fixes $A$ pointwise) and satisfies $\partial(s)=Ns$,
we conclude that the infinitesimal generator of $\beta$ is
$\partial=NsD$.
\end{proof}

\section{The nice case and a linearizability criterion}
\label{sec:linearization}

Throughout this section $B=k[x_1,\dots,x_n]$,
$D\in\LND(B)$ with slice $s$, $\partial=NsD$, and $\alpha$ is the
associated $\GM$-action on $\A^n=\Spec(B)$.

\begin{definition}
$D$ is \emph{nice} if $D^2(x_i)=0$ for all $i=1,\dots,n$.
The associated $\GM$-action is then also called \emph{nice}.
\end{definition}

\begin{remark}
If $D$ is nice then $\exp(-\lambda D)(x_i)=x_i-\lambda D(x_i)$,
so
\[
  \alpha_t(x_i) = x_i-(1-t^N)s\,D(x_i),\quad i=1,\dots,n.
\]
\end{remark}

The following lemma is the key algebraic input for the criterion.
Its proof uses the unique factorization in $k[x_1,\dots,x_n]$, and
we make this explicit to close a gap present in earlier arguments.

\begin{lemma}\label{lem:linear-factorization}
Let $E=\sum_{i=1}^n\lambda_i x_i\frac{\partial}{\partial x_i}$ be a
diagonal linear semisimple derivation of $k[x_1,\dots,x_n]$.
Suppose there exist $\delta\in\LND(B)$ and $\sigma\in B$ with
$E=N\sigma\delta$ and $\delta(\sigma)=1$.  Then, after a linear
change of coordinates,
\[
  E = Nx_n\frac{\partial}{\partial x_n},\quad
  \delta=\frac{\partial}{\partial x_n},\quad
  \sigma=x_n.
\]
\end{lemma}

\begin{proof}
Since $E=N\sigma\delta$, for each $i$ we have
$N\sigma\,\delta(x_i)=\lambda_i x_i$.
When $\lambda_i\neq 0$ this holds in the UFD $k[x_1,\dots,x_n]$,
so $\sigma$ divides $\lambda_i x_i$.  Since $x_i$ is irreducible
and $\lambda_i\in k^*$, we conclude $\sigma\mid x_i$.  As
$\delta(\sigma)=1$ and derivations vanish on constants,
$\sigma\notin k^*$, so $\sigma=c_i x_i$ for some $c_i\in k^*$.

If two indices $i\neq j$ had $\lambda_i\neq 0$ and
$\lambda_j\neq 0$, then $\sigma$ would divide both $x_i$ and
$x_j$.  Since $\gcd(x_i,x_j)=1$ in the UFD, this would force
$\sigma\in k^*$, a contradiction.  Hence exactly one eigenvalue
is nonzero; after reordering variables we may assume
$E=\lambda_n x_n\frac{\partial}{\partial x_n}$ — this is the linear
change of coordinates referred to in the statement.

From the above, $\sigma=cx_n$ for some $c\in k^*$.  Evaluating
$E(\sigma)$ in two ways: on one hand $E(\sigma)=\lambda_n cx_n$,
on the other $E(\sigma)=N\sigma\delta(\sigma)=N\sigma=Ncx_n$,
giving $\lambda_n=N$.  Then $Ncx_n\delta(x_n)=Nx_n$ yields
$\delta(x_n)=1/c$.  Replacing $x_n$ by $cx_n$ (another linear
change) gives the normal form $\sigma=x_n$, $\delta=\frac{\partial}{\partial x_n}$, $E=Nx_n\frac{\partial}{\partial x_n}$.
\end{proof}

\begin{theorem}\label{thm:linearizability-criterion}
The following are equivalent.
\begin{enumerate}
\item The $\GM$-action $\alpha$ is linearizable.
\item There exists $\phi\in\Aut(B)$ such that
  $\phi D\phi^{-1}=\frac{\partial}{\partial x_n}$ and $\phi(s)=x_n$.
\item There exist $\phi\in\Aut(B)$ and $p\in k[x_1,\dots,x_{n-1}]$
  such that $\phi D\phi^{-1}=\frac{\partial}{\partial x_n}$ and
  $\phi(s)=x_n+p(x_1,\dots,x_{n-1})$.
\end{enumerate}
\end{theorem}

\begin{proof}
\emph{$(2)\Rightarrow(3)$}: Immediate with $p=0$.

\smallskip
\emph{$(3)\Rightarrow(1)$}: Set $D'=\phi D\phi^{-1}=\frac{\partial}{\partial x_n}$
and $\sigma=\phi(s)=x_n+p$, so $\phi\partial\phi^{-1}=N\sigma D'=
N(x_n+p)\frac{\partial}{\partial x_n}$.  Define the triangular automorphism
$\tau$ by $\tau(x_i)=x_i$ for $i<n$ and $\tau(x_n)=x_n-p$; its
inverse satisfies $\tau^{-1}(x_n)=x_n+p$.  Since
$p\in k[x_1,\dots,x_{n-1}]$ does not involve $x_n$, for any
$g\in B$ one has
\[
  \tau\bigl[(x_n+p)\cdot g\circ\tau^{-1}\bigr]
  = \tau(x_n+p)\cdot g
  = (x_n-p+p)\cdot g
  = x_n\cdot g.
\]
Applying this with $g=\frac{\partial f}{\partial x_n}$ (using the chain
rule $\frac{\partial(\tau^{-1}f)}{\partial x_n}=\frac{\partial f}{\partial x_n}\circ\tau^{-1}$,
valid since $p$ is independent of $x_n$):
\[
  \Bigl[\tau\circ N(x_n+p)\tfrac{\partial}{\partial x_n}
    \circ\tau^{-1}\Bigr](f)
  = Nx_n\frac{\partial f}{\partial x_n}.
\]
Setting $\psi:=\tau\circ\phi$, we obtain
$\psi\partial\psi^{-1}=Nx_n\frac{\partial}{\partial x_n}$,
which is diagonal linear.  Hence $\alpha$ is linearizable.

\smallskip
\emph{$(1)\Rightarrow(2)$}: Let $\psi\in\Aut(B)$ linearize
$\alpha$, so $E:=\psi\partial\psi^{-1}$ is diagonal linear.  Set
$\delta:=\psi D\psi^{-1}$ and $\sigma:=\psi(s)$.  Then
$E=N\sigma\delta$ and $\delta(\sigma)=\psi(D(s))=1$.  By
Lemma~\ref{lem:linear-factorization}, composing $\psi$ with a
linear change of coordinates gives $\phi\in\Aut(B)$ with
$\phi D\phi^{-1}=\frac{\partial}{\partial x_n}$ and $\phi(s)=x_n$.
\end{proof}

\begin{remark}
Theorem~\ref{thm:linearizability-criterion} applies in particular
when $D$ is nice: condition~(3) then characterizes linearizability
of the nice $\GM$-action via $\phi D\phi^{-1}=\frac{\partial}{\partial x_n}$
and $\phi(s)=x_n+p$ for some $p\in k[x_1,\dots,x_{n-1}]$.  Note
that requiring $\phi(s)$ to have total degree $1$ in \emph{all}
variables is strictly stronger than necessary and would exclude
valid linearizable cases such as $\phi(s)=x_n+x_1^2$; the correct
condition is affine-linearity in $x_n$ alone, with coefficients in
$k[x_1,\dots,x_{n-1}]$ of arbitrary degree.
\end{remark}

\section{Further results and applications}\label{sec:applications}

\subsection{Independence of the choice of slice}

\begin{theorem}\label{thm:slice-independence}
Let $s$ and $s'=s+a$ ($a\in A:=\ker D$) be two slices of $D$,
and let $\partial=NsD$, $\partial'=Ns'D$.  Then
\[
  \phi_a\circ\partial\circ\phi_a^{-1} = \partial',\quad
  \text{where}\quad \phi_a:=\exp(-aD)\in\Aut(B).
\]
In particular, the linearizability of the $\GM$-action does not
depend on the choice of slice.
\end{theorem}

\begin{proof}
Since $a\in A=\ker(D)$, the exponential $\exp(-aD)$ is a
well-defined polynomial automorphism.  Using $D(s)=1$ and
$D^j(s)=0$ for $j\ge 2$:
\[
  \phi_a(s) = s - a\,D(s) = s - a = s'.
\]
Note that $\phi_a|_A=\id$, since $\exp(-aD)(b)=b$ for all
$b\in\ker D$ (as $D(b)=0$ makes every term beyond $j=0$ vanish).
To see $\phi_a D\phi_a^{-1}=D$: algebraically, $[D,D^j]=0$ for
all $j\ge 0$, so $D$ commutes with every term of $\exp(-aD)$,
giving $D\circ\exp(-aD)=\exp(-aD)\circ D$, i.e.,
$\phi_a D\phi_a^{-1}=D$.  Therefore:
\[
  \phi_a\circ(NsD)\circ\phi_a^{-1}
  = N\phi_a(s)\cdot\phi_a D\phi_a^{-1}
  = Ns'\cdot D = \partial'.\qedhere
\]
\end{proof}

\subsection{Geometric reformulation via the kernel}

\begin{corollary}\label{cor:kernel-criterion}
The $\GM$-action $\alpha$ is linearizable if and only if there
exists $\phi\in\Aut(B)$ such that
\[
  \phi(\ker D) = k[x_1,\dots,x_{n-1}]
  \quad\text{and}\quad
  \phi(s)\in x_n + k[x_1,\dots,x_{n-1}].
\]
\end{corollary}

\begin{proof}
Condition~(3) of Theorem~\ref{thm:linearizability-criterion}
requires $\phi D\phi^{-1}=\frac{\partial}{\partial x_n}$, whose kernel
is $k[x_1,\dots,x_{n-1}]$.  Since $\phi(\ker D)=\ker(\phi
D\phi^{-1})$, the result follows.
\end{proof}

The criterion says: the action is linearizable precisely when the
kernel of $D$ can be mapped onto a coordinate hyperplane ring and
the slice to an affine translate of the complementary variable.

\subsection{Complete answer for $n\le 3$}

We now apply the Main Theorem to resolve Freudenburg's question
completely in dimensions $n=1$, $2$, and $3$.  In each case the
key input is either a classical structure theorem for LNDs (the
additive side) or a linearization theorem for $\GM$-actions (the
multiplicative side), and the Main Theorem converts one into the
other.

For $n=1$ the statement is immediate: $k[x]$ admits only the
derivation $\frac{\partial}{\partial x}$ up to scalar, and every LND with
slice is already $\frac{\partial}{\partial x}$ itself.

For $n=2$, by Gutwirth's theorem~\cite{Gut62} every $\GM$-action on
$\A^2$ is linearizable, and by
Rentschler's theorem~\cite[Theorem~5.1.5]{Fre06} every LND of
$k[x,y]$ is locally triangularizable.  Combined with the Jung--van
der Kulk theorem~\cite[Theorem~5.1.10]{Fre06}, every LND of
$k[x,y]$ with slice is conjugate to $\frac{\partial}{\partial y}$ by a
tame automorphism, with the slice conjugated to $y+p(x)$.  By
Theorem~\ref{thm:linearizability-criterion}(3), the associated
$\GM$-action is linearizable.

For $n=3$ we give two independent proofs, reflecting the
additive--multiplicative duality of the problem.

\emph{Additive proof.}  By Wang's
classification~\cite[Proposition~4.6]{Wan05}
(Theorem~\ref{thm:wang-nice} above), every nice LND $D$ of
$k[x,y,z]$ satisfies $\rank(D)=1$ and $\ker D\cong k^{[2]}$.
Since $D$ admits a slice, the Slice Theorem gives
$B=\ker(D)[s]\cong k^{[3]}$, and in the resulting coordinate
system $D=\frac{\partial}{\partial z}$.  By
Theorem~\ref{thm:linearizability-criterion}(2), $\alpha$ is
linearizable.  The same argument, using Miyanishi's
theorem~\cite[Theorem~1.6]{Wan05} in place of
Theorem~\ref{thm:wang-nice}, extends to the strictly larger class
$D^2x=D^2y=0$ (Corollary~\ref{cor:wang-general} below).

\emph{Multiplicative proof (over $k=\mathbb{C}$).}  By the
linearization theorem of Kaliman, Koras, Makar-Limanov, and
Russell~\cite{KKMR97,KR97}, every algebraic $\C^*$-action on
$\C^3$ is linearizable.  In particular $\alpha$ is linearizable,
and Theorem~\ref{thm:linearizability-criterion} then yields
$\phi D\phi^{-1}=\frac{\partial}{\partial z}$.

\begin{remark}
The two proofs for $n=3$ are not redundant: Wang's applies over
any algebraically closed field of characteristic zero, while KKMR
gives the stronger statement that \emph{every} $\GM$-action on
$\A^3$ is linearizable (not only nice ones arising from LNDs with
slice).  The proof of the KKMR theorem required building the list
of Koras--Russell threefolds and distinguishing them from $\C^3$
using the Makar-Limanov invariant; Wang's proof is a direct
algebraic argument with filtrations and degree functions.
\end{remark}

We state the results for $n=2$ and $n=3$ formally for reference.

\begin{corollary}\label{cor:dim2}
Let $B=k[x,y]$ and $D\in\LND(B)$ nice with slice $s$.  Then the
associated $\GM$-action on $\A^2$ is linearizable, and $D$ is
conjugate to $\frac{\partial}{\partial y}$.
\end{corollary}

\begin{proof}
By Rentschler's theorem~\cite[Theorem~5.1.5]{Fre06} and the
Jung--van der Kulk theorem~\cite[Theorem~5.1.10]{Fre06}, $D$ is
conjugate to $\frac{\partial}{\partial y}$ with slice conjugated to
$y+p(x)$.  Apply Theorem~\ref{thm:linearizability-criterion}(3).
\end{proof}

\begin{theorem}[Wang~{\cite[Proposition~4.6]{Wan05}}]
\label{thm:wang-nice}
Let $D\in\LND(k[x,y,z])$ be nonzero with $D^2x=D^2y=D^2z=0$.
Then $\ker D$ contains a nonzero linear form in $\{x,y,z\}$ and
$\rank(D)=1$.  If moreover $D$ is irreducible, there exists a
coordinate system $(x',y',z')$ of $k[x,y,z]$ related to
$(x,y,z)$ by a linear change of variables such that
\[
  D = f(x')\frac{\partial}{\partial y'} +
      g(x')\frac{\partial}{\partial z'}
\]
with $f,g\in k[x']$ and $\gcd(f,g)=1$.  In all cases
$\ker D\cong k^{[2]}$.
\end{theorem}

\begin{theorem}[Wang~{\cite[Theorem~4.7]{Wan05}}]
\label{thm:wang-semi-nice}
Let $D\in\LND(k[x,y,z])$ be irreducible with $D^2x=D^2y=0$.
Then one of the following holds:
\begin{enumerate}
\item There exists a coordinate system $(L_1,L_2,z)$ of
  $k[x,y,z]$, with $L_1,L_2$ linear forms in $x,y$, such that
  $D(L_1)=0$, $D(L_2)\in k[L_1]$, and $D(z)\in k[L_1,L_2]$.
\item There exists a coordinate system $(v,x,y)$ such that
  $D(v)=0$ and $Dx,Dy\in k[v]$.  In this case $D$ has a slice
  and $\rank(D)=1$.
\end{enumerate}
\end{theorem}

\begin{corollary}\label{cor:dim3-wang}
Let $D\in\LND(k[x,y,z])$ be nice and admit a slice $s$.  Then
$D$ is automorphically conjugate to $\frac{\partial}{\partial z}$, and
the associated $\GM$-action on $\A^3$ is linearizable.
\end{corollary}

\begin{proof}
By Theorem~\ref{thm:wang-nice}, $\rank(D)=1$ and
$\ker D\cong k^{[2]}$.  By the Slice Theorem, $B=\ker(D)[s]$, so
identifying $s$ as the free variable over $\ker D=k[x_2,x_3]$
(via the universal property of polynomial rings) gives a
coordinate system in which $D=\frac{d}{ds}$; relabeling $s=z$
gives $D=\frac{\partial}{\partial z}$.
Condition~(2) of Theorem~\ref{thm:linearizability-criterion}
holds with $\phi=\id$.
\end{proof}

\begin{corollary}\label{cor:wang-general}
Let $D\in\LND(k[x,y,z])$ be irreducible with $D^2x=D^2y=0$ and
admit a slice.  Then $D$ is conjugate to $\frac{\partial}{\partial z}$,
and the associated $\GM$-action is linearizable.
\end{corollary}

\begin{proof}
Let $D\in\LND(k[x,y,z])$ be irreducible with $D^2x=D^2y=0$ and
admit a slice $s$.  By the Slice Theorem,
$B=\ker(D)[s]$.  By Miyanishi's theorem~\cite[Theorem~1.6]{Wan05},
$\ker(D)\cong k^{[2]}$, so $\ker D=k[u,v]$ for some $u,v\in B$,
and $B=k[u,v][s]\cong k^{[3]}$.  In the coordinate system
$(u,v,s)$, $D$ acts as $\frac{\partial}{\partial s}$, i.e., $D$ is the
standard derivation in $s$.  Renaming $s=z$ (and relabeling),
we obtain $D=\frac{\partial}{\partial z}$ in the coordinates $(u,v,z)$.
This is condition~(2) of
Theorem~\ref{thm:linearizability-criterion}, so the associated
$\GM$-action is linearizable.

Note that this argument uses only the existence of a slice and
Miyanishi's theorem; it does not require further case analysis
from Theorem~\ref{thm:wang-semi-nice}.
\end{proof}

\subsection{State of the problem in dimension $n\ge 4$}

\begin{remark}
The linearization problem for $\GM$-actions on $\A^n$ with
$n\ge 4$ is open.  The present work shows that any
counterexample to Freudenburg's question, if it exists, must
arise from a LND with slice on $k[x_1,\dots,x_n]$, $n\ge 4$,
whose kernel is \emph{not} isomorphic to a coordinate hyperplane
ring in a way compatible with the slice.  Wang's methods apply
in higher dimensions only under additional assumptions on the
second compositional power $D^2$ of the derivation.
\end{remark}

\subsection{An obstruction to linearizability via the
Makar-Limanov invariant}

For potential applications in dimension $n\ge 4$, we record a
necessary condition for linearizability based on the
Makar-Limanov invariant, which is the same tool used in the KKMR
proof for $n=3$.

\begin{definition}
The \emph{Makar-Limanov invariant} of $B$ is
\[
  \ML(B) := \bigcap_{D\in\LND(B)}\ker(D).
\]
\end{definition}

\begin{proposition}\label{prop:ML-obstruction}
If the $\GM$-action $\alpha$ associated to $(D,s)$ is linearizable,
then $\ML(B)=k$.
\end{proposition}

\begin{proof}
A linearizable action means there exists $\phi\in\Aut(B)$ such
that $\phi\partial\phi^{-1}$ is a diagonal linear derivation.
In particular $\Spec(B)\cong\A^n$ as varieties.  Since
$\ML(k[x_1,\dots,x_n])=k$ (the LNDs $\frac{\partial}{\partial x_i}$
have trivial intersection of kernels) and $\ML$ is an isomorphism
invariant, we get $\ML(B)=k$.

Note that $\ML(B)=k$ is not automatic: while it holds for
polynomial rings, a finitely generated $k$-domain $B$ may have
$\ML(B)\supsetneq k$ (e.g., Danielewski surfaces or the
Koras--Russell threefolds), precisely because such $B$ are not
isomorphic to affine space.
\end{proof}

\begin{remark}
Proposition~\ref{prop:ML-obstruction} is useful as follows: to
show that a given nice $\GM$-action on $\Spec(B)$ is
\emph{not} linearizable for $n\ge 4$, it suffices to exhibit a
nonscalar element of $B$ lying in $\ker(D')$ for every LND $D'$
of $B$.  Such an element witnesses $\ML(B)\supsetneq k$ and
rules out linearizability without requiring a complete
classification of $\GM$-actions.  This is precisely the strategy
that Makar-Limanov used to prove $\Spec(B)\not\cong\A^3$ for the
Koras--Russell threefolds.
\end{remark}


\end{document}